\documentclass[12pt,reqno]{amsart}
\usepackage{amsmath} 

\usepackage{amsthm}
\usepackage{amssymb}
\usepackage{graphics}
\usepackage{latexsym}

\numberwithin{equation}{section}
\newtheorem{thm}{Theorem}[section]
\newtheorem{prop}[thm]{Proposition}
\newtheorem{lem}[thm]{Lemma}
\newtheorem{cor}[thm]{Corollary}

\theoremstyle{remark}

\newtheorem{defn}{Definition}

\newcommand{\BBB}{\mathbb}
\newcommand{\R}{{\BBB R}}
\newcommand{\Z}{{\BBB Z}}
\newcommand{\T}{{\BBB T}}

\newcommand{\N}{{\BBB N}}

\newcommand{\HT}{{\mathcal H}}%
\newcommand{\LR}[1]{{\langle {#1} \rangle }}

\newcommand{\lec}{{\ \lesssim \ }}

\newcommand{\al}{\alpha}

\newcommand{\ga}{\gamma}
\newcommand{\Ga}{\Gamma}

\newcommand{\vp}{\varphi}

\newcommand{\e}{\varepsilon}

\newcommand{\la}{\lambda}
\newcommand{\La}{\Lambda}

\newcommand{\de}{\delta}

\newcommand{\supp}{\operatorname{supp}}

\newcommand{\imply}{\Rightarrow}
\newcommand{\I}{\infty}

\newcommand{\sgn}{\operatorname{sgn}}

\newcommand{\EQS}[1]{\begin{align} #1 \end{align}}
\newcommand{\EQQS}[1]{\begin{align*} #1 \end{align*}}
\newcommand{\EQQ}[1]{\begin{equation*} \begin{split} #1
 \end{split} \end{equation*}}

\newcommand{\F}{\mathcal{F}}

\newcommand{\ti}{\widetilde}
\newcommand{\ha}{\widehat}

\newcommand{\ds}{\partial_{x}}%
\newcommand{\dt}{\partial_{t}}%
\title[L.W.P. for Benjamin-Ono type equations]
{Local well-posedness for third order Benjamin-Ono type equations on the torus}

\author[T. Tanaka]{Tomoyuki Tanaka}
\address[T. Tanaka]{Graduate School of Mathematics, Nagoya University,
Chikusa-ku, Nagoya, 464-8602, Japan}
\email[T. Tanaka]{d18003s@math.nagoya-u.ac.jp}

\begin{document}

\begin{abstract}
We consider the Cauchy problem of third order Benjamin-Ono type equations on the torus.
Nonlinear terms may yield derivative losses, which prevents us from using the classical energy method.
In order to overcome that difficulty, we add a correction term into the energy.
We also use the Bona-Smith type argument to show the continuous dependence.
\end{abstract}
\maketitle
\setcounter{page}{001}


\section{Introduction}


We consider the Cauthy problem of the following third order Benjamin-Ono type equations on the torus $\T(:=\R/2\pi\Z)$:
\EQS{
\label{BO1}
&\dt u-\ds^{3}u+u^{2}\ds u+c_{1}\ds(u\HT\ds u)+c_{2}\HT\ds(u\ds u)=0,\quad(t,x)\in\R\times\T,\\
\label{BO2}
&u(0,x)=\vp(x),
}
where the initial data $\vp$ and the unknown function $u$ are real valued, and $c_{1},c_{2}\in\R$.
$\HT$ is the Hilbert transform on the torus defined by
\[\ha{\HT f}(0)=0\quad{\rm and}\quad\ha{\HT f}(k)=-i\sgn(k)\hat{f}(k),\quad k\in\Z\backslash\{0\},\]
where $\hat{f}$ is the Fourier transform of $f$: $\hat{f}(k)=\F f(k)=(2\pi)^{-1/2}\int_{\T}f(x)e^{-ixk}dx$.
The well-known Benjamin-Ono equation
\EQS{
\label{2BO}
\dt u+\HT\ds^{2}u+2u\ds u=0
}
describes the behavior of long internal waves in deep stratified fluids.
The equation (\ref{2BO}) also has infinitely many conservation laws,
which generates a hierarchy of Hamiltonian equations of order $j$.
The equation (\ref{BO1}) with $c_{1}=c_{2}=\sqrt{3}/2$
is the second equation in the Benjamin-Ono hierarchy \cite{Matsuno}.

There are a lot of literature on the Cauchy problem on (\ref{2BO}).
On the real line case,
Ionescu-Kenig \cite{IoKe07} showed the local well-posedness in $H^{s}(\R)$ for $s\ge0$
(see also \cite{Molinet12} for another proof
and \cite{IoKe072} for the local well-posedness with small complex valued data).
On the periodic case, Molinet \cite{Molinet08,Molinet09} showed the local well-posedness
in $H^{s}(\T)$ for $s\ge0$ and that this result was sharp.
See \cite{ABFS,BuPl,Iorio86,KeKe,KoTz,Ponce,Tao} for former results.

On the Cauchy problem of (\ref{BO1}) with $c_{1}=c_{2}=\sqrt{3}/2$ on the real line,
Feng-Han \cite{FeHa96} proved the unique existence in $H^{s}(\R)$ for $4\le s\in\N$
by using the theory of complete integrability.
They also used the energy method with a correction term in order to show the uniqueness.
Feng \cite{Feng97} modified the energy method used in \cite{FeHa96}
and used an {\it a priori} bound of solutions in $H^{s}(\R)$
to show the ``weak'' continuous dependence in the following sense:
\EQS{\label{wcd}
\vp_{n}\to\vp\ {\rm in}\ H^{s-2}(\R)\ {\rm as}\ n\to\I
  \imply u_{n}\to u\ {\rm in}\ C([0,T];H^{s-2}(\R))\ {\rm as}\ n\to\I,
}
for $\vp,\vp_{n}\in H^{s}(\R)$ and $6\le s\in\N$.
Here, $u_{n}$ (resp. $u$) denotes the corresponding solution of (\ref{BO1})
with $c_{1}=c_{2}=\sqrt{3}/2$
and the initial data $\vp_{n}$ for $n\in\N$ (resp. $\vp$).
Note that the topology of the convergence is weaker than $H^{s}$.
Linares-Pilod-Ponce \cite{LPP11} and Molinet-Pilod \cite{MP12} succeed in proving
the local well-posedness in $H^{s}(\R)$ of the following equation
\[\dt u+d_{1}\ds^{3}u-d_{2}\HT\ds^{2}u=d_{3}u\ds u-d_{4}\ds(u\HT\ds u+\HT(u\ds u)),\]
for $s\ge2$ and $s\ge1$,
respectively.
Here, coefficients satisfy $d_{1}\in\R$, $d_{1}\neq0$ and $d_{j}>0$ for $j=2,3,4$.
Their proof involves the gauge transform and the Kato type smoothing estimate.

On the periodic case, as far as the author knows, there are no well-posedness results for the Cauchy problem of
(\ref{BO1}) available in the literature.
Although proofs in Feng-Han \cite{FeHa96} and Feng \cite{Feng97} above works,
and we cannot obtain the local well-posedness, that is, the resultant continuous dependence (\ref{wcd}) is weak.
And their proofs heavily depend on the complete integrability.
In particular, it is very important to have $c_{1}=c_{2}$ in their proofs.
It should also be pointed out that in the periodic case, we do not have the Kato type smoothing estimate,
which implies that the local well-posedness is far from trivial.

Therefore, in this article, we are interested in establishing
the local well-posedness of (\ref{BO1}) in $H^{s}(\T)$ for $s$ less than $4$
without using the theory of complete integrability.
In particular, we improve the ``weak'' continuous dependence (\ref{wcd}) shown in \cite{Feng97}
in order to fulfill conditions of the local well-posedness.
Moreover, thanks to Lemma \ref{comm.est.4}, we can show the local well-posedness of the non-integrable case (\ref{BO1}).

The main result is the following:

\begin{thm}\label{main}
Let $s\ge s_{0}>5/2$.
For any $\vp\in H^{s}(\T)$,
there exist $T=T(\|\vp\|_{H^{s_{0}}})>0$ and the unique solution $u\in C([-T,T];H^{s}(\T))$
to the IVP (\ref{BO1})--(\ref{BO2}) on $[-T,T]$.
Moreover, for any $R>0$,
the solution map $\vp\mapsto u(t)$ is continuous
from the ball $\{\vp\in H^{s}(\T);\|\vp\|_{H^{s}}\le R\}$ to $C([-T,T];H^{s}(\T))$.
\end{thm}

Now, we mention the idea of the proof of Theorem \ref{main}.
The standard energy method gives us the local well-posedness of (\ref{2BO}) in $H^{s}(\T)$ for $s>3/2$.
On the other hand, nonlinear terms $\ds(u\HT\ds u)$ and $\HT\ds(u\ds u)$ in (\ref{BO1}) have two derivatives,
and the energy estimate gives only the following:
\EQS{\label{deri.los.}
\frac{d}{dt}\|\ds^{k}u(t)\|_{L^{2}}^{2}
  \lesssim(1+\|\ds^{2}u\|_{L^{\I}})^{2}\|\ds^{k}u(t)\|_{L^{2}}^{2}+\left|\int\ds u(\HT\ds^{k+1}u)\ds^{k}udx\right|.
}
It is difficult to handle the last term in the right hand side by $\|u\|_{H^{k}}$,
which is the main difficulty in this problem.
To overcome that difficulty, we add a correction term into the energy (see Definition \ref{def1}):
\EQQS{
E_{*}(u)
:=\|u\|_{L^{2}}^{2}+\|D^{s}u\|_{L^{2}}^{2}+a_{s}\|u\|_{L^{2}}^{4s+2}
  +b_{s}\int u(\HT D^{s}u)D^{s-2}\ds udx,
}
where $D:=\F^{-1}|\xi|\F$, following the idea from Kwon \cite{Kwon},
who studied the local well-posedness of the fifth order KdV equation
(see also Segata \cite{Segata}, Kenig-Pilod \cite{KePi16} and Tsugawa \cite{Tsugawa17}).
The correction term allows us to cancel out the worst term in (\ref{deri.los.}),
which makes it possible to evaluate the $H^{s}$-norm of the solution by that of the initial data.
It is worth pointing out that our proof refines the idea in \cite{Feng97}.
Indeed, Feng introduced the following energy estimate
in order to show the ``weak'' continuous dependence (\ref{wcd}):
\EQQS{
&\frac{d}{dt}\left(\|\ds^{k-2}w\|_{L^{2}}^{2}+\frac{2k-3}{4}\int_{\R}(u+v)\ds^{k-3}w\HT\ds^{k-2}wdx\right)\\
&\le C(T,\|\vp\|_{H^{k}},\|\psi\|_{H^{k}})\|w(t)\|_{H^{k-2}}^{2},
}
on $[0,T]$,
where $w=u-v$ and $u,v\in C([0,T];H^{k}(\R))$ satisty (\ref{BO1}) with $c_{1}=c_{2}=\sqrt{3}/2$
and initial data $\vp,\psi\in H^{k}(\R)$, respectively.
Here, we would like to have the estimate for $\|w\|_{H^{k}}$.
If we simply replace $k-2$ with $k$ in the above estimate,
the constant in the right hand side depends on $\|\vp\|_{H^{k+2}}$ (resp. $\|\psi\|_{H^{k+2}}$),
which cannot be handled by $\|\vp\|_{H^{k}}$ (resp. $\|\psi\|_{H^{k}}$).
Therefore, we need to find a different correction term (see Definition \ref{def1})
and estimate the difference between two solutions in $H^{k}(\T)$ more carefully
(see the proof of Proposition \ref{en.dif.s2})
so as to complete the continuous dependence.

Another difficulty is the presence of the Hilbert transform $\HT$,
which restricts the possibility of using the integration by parts for some terms.
Recall that for real valued functions $f,g$, we have
\[|\LR{fD^{s}g,D^{s}\ds g}_{L^{2}}|\le\frac{1}{2}\|\ds f\|_{\I}\|D^{s}g\|_{L^{2}}^{2}.\]
However, in our problem we cannot apply the integration by parts to
\[\LR{\ds f\HT D^{s}\ds g,D^{s}g}_{L^{2}},\]
which is nothing but the term which we cancel out by introducing a correction term.


We notice that the $L^{2}$-norm is conserved by solutions of equations (\ref{BO1}) with $c_{1}=c_{2}$
thanks to the following equality:
\[\LR{\HT\ds(u\ds u),u}_{L^{2}}+\LR{\ds(u\HT\ds u),u}_{L^{2}}=0,\]
which helps us to handle nonlinear terms.
In the case $c_{1}\neq c_{2}$, we use Lemma \ref{comm.est.4} originally proved in \cite{DMcP}.

Subsequently, using the conservation law corresponding to the $H^{3}$-norm of the solution,
we can obtain the following result:

\begin{cor}
The Cauchy problem (\ref{BO1})--(\ref{BO2}) with $c_{1}=c_{2}=\sqrt{3}/2$
is globally well-posed in $H^{s}(\T)$ for $s\ge3$.
\end{cor}

This paper is organized as follows.
In Section 2, we fix some notations and state a number of estimates.
We also obtain a solution of the regularlized equation associated to (\ref{BO1}).
In Section 3, we give an {\it a priori} estimate for the solution to (\ref{BO1}).
In Section 4, we show the existence of the solution, uniqueness, the persistence, and the continuous dependence.

\section{Notations, Preliminaries and Parabolic Regularization}

In this section, we give some notations and collect a number of estimates which will be used throughout this paper.
We denote the norm in $L^{p}(\T)$ by $\|\cdot\|_{p}$.
In particular, we simply write $\|\cdot\|:=\|\cdot\|_{2}$.
We denote $\|f\|_{H^{s}}:=2^{-1/2}(\|f\|^{2}+\|D^{s}f\|^{2})^{1/2}$ for a function $f$ and $s\ge0$,
where $D=\F^{-1}|\xi|\F$.
Let $\LR{\cdot,\cdot}:=\LR{\cdot,\cdot}_{L^{2}}$.
We also use the same symbol for $\LR{\cdot}:=(1+|\cdot|^{2})^{1/2}$.
Let $[A,B]=AB-BA$.

We use the following Gagliardo-Nirenberg inequality on the torus:

\begin{lem}\label{G.N.}
Assume that $l\in\mathbb{N}\cup \{0\}$ and $s\ge1$ satisfy $l\le s-1$ and a real number $p$ satisfies $2\le p\le\I$.
Put $\al=(l+1/2-1/p)/s$.
Then, we have
\begin{equation*}
\|\ds^{l}f\|_{p}\lesssim
\begin{cases}
\|f\|^{1-\al}\|D^{s}f\|^{\al}\quad (when\quad 1\le l\le s-1),\\
\|f\|^{1-\al}\|D^{s}f\|^{\al}+\|f\|\quad (when\quad l=0),
\end{cases}
\end{equation*}
for any $f\in H^{s}(\T)$.
\end{lem}

\begin{proof}
In the case $s$ is an integer, see Section 2 in \cite{SchwarzJr.}.
The general case follows from the integer case and the H\"older inequality.
\end{proof}

The following inequality is helpful when we estimate the difference between two solutions in $L^{2}$.

\begin{lem}\label{freq.est.}
Let $k\in\N\cup\{0\}$.
Then the following inequality holds true:
\[\|\HT\ds^{k}f+\LR{D}^{-1}\ds^{k+1}f\|\le\|f\|_{H^{k-1}}\]
for any $f\in H^{k-1}(\T)$.
\end{lem}

\begin{proof}
We have $|\sgn(\xi)-\xi\LR{\xi}^{-1}|
  \le\LR{\xi}^{-1}$ for any $\xi\in\Z,$
which shows that
\EQQS{
\|\HT\ds^{k}f+\LR{D}^{-1}\ds^{k+1}f\|
=\|(\sgn(\xi)-\xi\LR{\xi}^{-1})\xi^{k}\hat{f}(\xi)\|_{l^{2}}\le\|f\|_{H^{k-1}}
}
as desired.
\end{proof}


\begin{defn}
For $s\ge0$ and functions $u,v$ defined on $\T$, we define
\EQQS{
&P_{s}(f,g):=D^{s}\ds(f\ds g)-D^{s}\ds f\ds g-\ds fD^{s}\ds^{2}g-(s+1)\ds fD^{s}\ds g,\\
&Q_{s}(f,g):=\HT D^{s}\ds(f\ds g)-(\HT D^{s}\ds f)\ds g-\ds f\HT D^{s}\ds^{2}g-(s+1)\ds f\HT D^{s}\ds g.
}
\end{defn}

We introduce several commutator estimates.
For general theory on the real line, see \cite{FGO}.
We shall use extensively the following commutator estimate.

\begin{lem}\label{comm.est.}
Let $s\ge1$ and $s_{0}>5/2$.
Then there exists $C=C(s,s_{0})>0$ such that for any $f,g\in H^{s}(\T)\cap H^{s_{0}}(\T)$,
\EQQS{
\|P_{s}(f,g)\|,\|Q_{s}(f,g)\|
\le C(\|f\|_{H^{s_{0}}}\|g\|_{H^{s}}+\|f\|_{H^{s}}\|g\|_{H^{s_{0}}}).
}
\end{lem}

\begin{proof}
We show only the inequality for $P_{s}(f,g)$ with $s>1$.
The case $s=1$ follows from Lemma \ref{comm.est.4}.
Another one follows from a similar argument since $D=\HT\ds$.
It suffices to show that there exists $C=C(s)$ such that
\EQS{\label{eq2.1}
\begin{aligned}
&||\xi|^{s}\xi\eta-|\xi-\eta|^{s}(\xi-\eta)\eta-|\eta|^{s}\eta^{2}-(s+1)(\xi-\eta)|\eta|^{s}\eta|\\
&\le C(|\xi-\eta|^{s}|\eta|^{2}+|\xi-\eta|^{2}|\eta|^{s})
\end{aligned}
}
for any $\xi,\eta\in\Z$.
We split the summation region into three regions:
$R_{1}=\{3|\eta|\le|\xi-\eta|\}$, $R_{2}=\{|\eta|\ge3|\xi-\eta|\}$
and $R_{3}=\{|\xi-\eta|/4\le|\eta|\le4|\xi-\eta|\}$.
On $R_{1}$, the mean value theorem shows that (\ref{eq2.1}) holds.
On $R_{2}$, note that $|\xi|\sim|\eta|$.
It immediately follows that $|\xi-\eta|^{s}(\xi-\eta)\eta\lesssim|\xi-\eta|^{s}|\eta|^{2}$.
Set $\sigma(x)=x|x|^{s}$ for $x\in\mathbb{R}$.
Note that $\sigma\in C^{2}(\R)$.
The Taylor theorem shows that there exist $\tilde{\eta}\in(\xi,\eta)$ or $\tilde{\eta}\in(\eta,\xi)$ such that
\[\sigma(\xi)=\sigma(\eta)+\sigma'(\eta)(\xi-\eta)+\frac{\sigma''(\tilde{\eta})}{2}(\xi-\eta)^{2}.\]
This together with the fact that $|\tilde{\eta}|\sim|\xi|\sim|\eta|$ implies that (\ref{eq2.1}) holds.
On $R_{3}$, it is obvious.
\end{proof}

\begin{lem}\label{comm.est.2}
Let $s\ge1$, $s_{0}>1/2$ and $\La_{s}=D^{s}$ or $D^{s-1}\ds$.
Then we have the following:\\
(i) There exists $C(s,s_{0})>0$ such that for any $f,g\in H^{s_{0}+1}(\T)\cap H^{s}(\T)$,
\[\|[\La_{s},f]\ds g\|\lesssim\|f\|_{H^{s_{0}+1}}\|g\|_{H^{s}}+\|f\|_{H^{s}}\|g\|_{H^{s_{0}+1}}.\]
(ii) There exists $C(s_{0})>0$ such that for any $f\in H^{s_{0}+1}(\T)$ and $g\in L^{2}(\T)$,
\[\|[\LR{D}^{-1}\La_{2},f]g\|\lesssim\|f\|_{H^{s_{0}+1}}\|g\|.\]
\end{lem}

\begin{proof}
We omit the proof of the $(i)$ since it is identical with that of the previous Lemma.
We show the case $(ii)$ with $\La_{2}=\ds^{2}$ only.
The other case follows from a similar argument.
It suffices to show that
$|\xi^{2}\LR{\xi}^{-1}-\eta^{2}\LR{\eta}^{-1}|\lesssim|\xi-\eta|$ for any $\xi,\eta\in\Z$.
Set $\sigma(x)=-x^{2}\LR{x}^{-1}$ for $x\in\R$.
Note that $\sigma\in C^{1}(\R)$ and that $\sigma'(x)=-(x^{3}+2x)\LR{x}^{-3}$.
It then follows that there exists $C>0$ such that $|\sigma'(x)|\le C$ for any $x\in\R$.
This together with the mean value theorem implies that we have
\[|\sigma(\xi)-\sigma(\eta)|\le C|\xi-\eta|,\]
which completes the proof.
\end{proof}

The following estimate is essential for our analysis in the case $c_{1}\neq c_{2}$ in (\ref{BO1}).
For $L^{p}$ cases on the real line, see \cite{DMcP}.

\begin{lem}\label{comm.est.4}
Let $s_{0}>1/2$ and $k\in\N$.
Then, there exists $C=C(s_{0})>0$ such that for any $f\in H^{s_{0}}(\T)$ and $g\in L^{2}(\T)$
\[\|[\HT,f]\ds^{k}g\|\le C\|f\|_{H^{s_{0}+k}}\|g\|.\]
\end{lem}

\begin{proof}
It suffices to show that
\EQS{\label{con.H.comm.}
|\sgn(\xi)-\sgn(\eta)||\eta|^{k}\lesssim|\xi-\eta|^{k}
}
for any $\xi,\eta\in\Z$.
We split the summation region into three regions:
$R_{1}=\{3|\eta|\le|\xi|\}$, $R_{2}=\{|\eta|\ge3|\xi|\}$
and $R_{3}=\{|\xi|/4\le|\eta|\le4|\xi|\}$.
It is clear that (\ref{con.H.comm.}) holds on $R_{1}$ and $R_{2}$.
It is also clear that (\ref{con.H.comm.}) holds when $\xi\eta>0$.
Therefore, we consider the region $R_{3}\cap\{\xi\eta\le0\}$.
We first assume that $\xi\ge0$ and $\eta\le0$.
Note that $|\xi-\eta|\ge|\xi|\ge|\eta|/4$.
Similarly, in the case $\xi\le0$ and $\eta\ge0$ we have $|\xi-\eta|\ge|\eta|$.
Therefore, we have (\ref{con.H.comm.}), which concludes the proof.
\end{proof}

\begin{lem}\label{reduction}
Let $s_{0}>1/2$ and $u,v$ be sufficiently smooth function defined on $\T$.
Then there exists $C=C(s_{0})>0$ such that
\[|\LR{v\HT\ds^{2}u+\ds v\HT\ds u,u}|\le C\|v\|_{H^{s_{0}+2}}\|u\|^{2}.\]
\end{lem}

\begin{proof}
This follows from the equality
\[2\LR{v\HT\ds^{2}u+\ds v\HT\ds u,u}=-\LR{[\HT,v]\ds^{2}u,u}-\LR{\ds^{2}v\HT u,u}\]
together with Lemma \ref{comm.est.4}.
\end{proof}

We shall also use extensively the following estimate.

\begin{lem}\label{IBP}
Let $s_{0}>1/2$.
Then, there exists $C=C(s_{0})>0$ such that for any $f\in H^{s_{0}+1}(\T)$ and $g\in H^{1}(\T)$
\[|\LR{f\ds g,g}|\le C\|f\|_{H^{s_{0}+1}}\|g\|^{2}.\]
\end{lem}

\begin{proof}
This follows from the density argument and the integration by parts.
\end{proof}


The following lemma helps us calculate a correction term.

\begin{lem}\label{good1}
For sufficiently smooth functions $f,g$ and $h$ defined on $\T$, it holds that
\EQQS{
\LR{\ds^{3}fg,h}+\LR{f\ds^{3}g,h}+\LR{fg,\ds^{3}h}=3\LR{\ds f\ds g,\ds h}
.
}
\end{lem}

\begin{proof}
See Lemma 2.2 in \cite{KePi16}.
\end{proof}

We shall repeatedly use estimates of the following type:

\begin{lem}\label{good2}
Let $s_{0}>5/2$.\\
(i) Let $s\ge1$.
There exists $C(s,s_{0})>0$ such that
for any $f_{1}\in H^{s}(\T)\cap H^{s_{0}}(\T)$ and $f_{2}\in H^{s+1}(\T)\cap H^{s_{0}}(\T)$,
\EQQS{
|\LR{f_{1}\HT D^{s}f_{2},\HT D^{s}(f_{1}\ds f_{2})}|
\le C(\|f_{1}\|_{H^{s_{0}}}^{2}\|f_{2}\|_{H^{s}}^{2}
      +\|f_{1}\|_{H^{s_{0}}}\|f_{1}\|_{H^{s}}\|f_{2}\|_{H^{s_{0}}}\|f_{2}\|_{H^{s}}).
}
(ii) Let $s\ge2$.
There exists $C(s,s_{0})>0$ such that
for any $f_{1}\in H^{s+1}(\T)\cap H^{s_{0}}(\T)$ and $f_{2}\in H^{s+2}(\T)\cap H^{s_{0}}(\T)$,
\EQQS{
&|\LR{f_{1}\HT D^{s}\ds(f_{1}\HT\ds f_{2}),D^{s-2}\ds f_{2}}|\\
&\le C(\|f_{1}\|_{H^{s_{0}}}^{2}\|f_{2}\|_{H^{s}}^{2}
      +\|f_{1}\|_{H^{s_{0}}}\|f_{1}\|_{H^{s}}\|f_{2}\|_{H^{s_{0}}}\|f_{2}\|_{H^{s}}).}
\end{lem}

\begin{proof}
First we show $(i)$.
Note that
\EQQS{
|\LR{f_{1}\HT D^{s}f_{2},\HT D^{s}(f_{1}\ds f_{2})}|
\le|\LR{f_{1}\HT D^{s}f_{2},[\HT D^{s},f_{1}]\ds f_{2}}|+|\LR{f_{1}^{2}\HT D^{s}f_{2},\HT D^{s}\ds f_{2}}|.
}
Lemma \ref{IBP} together with $(i)$ of Lemma \ref{comm.est.2} shows $(i)$.
Next we show $(ii)$.
Lemma \ref{comm.est.} shows that
\EQQS{
&|\LR{D^{s+1}(f_{1}\HT\ds f_{2}),f_{1}D^{s-2}\ds f_{2}}-R_{1}-R_{2}-R_{3}|\\
&\lesssim\|f_{1}\|_{H^{s_{0}}}^{2}\|f_{2}\|_{H^{s}}^{2}
          +\|f_{1}\|_{H^{s_{0}}}\|f_{1}\|_{H^{s}}\|f_{2}\|_{H^{s_{0}}}\|f_{2}\|_{H^{s}},
}
where $R_{1}=\LR{D^{s+1}f_{1}\HT\ds f_{2},f_{1}D^{s-2}\ds f_{2}}$,
$R_{2}=\LR{f_{1}\HT D^{s+1}\ds f_{2},f_{1}D^{s-2}\ds f_{2}}$ and
$R_{3}=(s+1)\LR{\ds f_{1}\HT D^{s+1}f_{2},f_{1}D^{s-2}\ds f_{2}}$.
It is easy to see that
\[|R_{1}|\lesssim\|f_{1}\|_{H^{s_{0}}}\|f_{1}\|_{H^{s}}\|f_{2}\|_{H^{s_{0}}}\|f_{2}\|_{H^{s}}
\quad\text{and}\quad
|R_{3}|\lesssim\|f_{1}\|_{H^{s_{0}}}^{2}\|f_{2}\|_{H^{s}}^{2}.\]
For $R_{2}$, we have
\EQQS{
R_{2}
&=-\LR{f_{1}^{2}D^{s}\ds^{2}f_{2},D^{s-2}\ds f_{2}}
=2\LR{f\ds f_{1}D^{s}\ds f_{2},D^{s-2}\ds f_{2}}-\LR{f_{1}^{2}D^{s}\ds f_{2},D^{s}f}\\
&=-2\LR{\ds(f\ds f_{1}D^{s-2}\ds f_{2}),D^{s}f_{2}}+\LR{f_{1}\ds f_{1},(D^{s}f_{2})},
}
which can be bounded by $\lesssim\|f_{1}\|_{H^{s_{0}}}^{2}\|f_{2}\|_{H^{s}}^{2}$.
This concludes the proof.
\end{proof}

\begin{lem}\label{DL1}
For any $s\ge1$ and $s_{0}>5/2$,
there exists $C(s,s_{0})>0$ such that
for any $u,v\in H^{s+2}(\T)\cap H^{s_{0}}(\T)$,
\EQQS{
&|\LR{D^{s}\ds(u\HT\ds u-v\HT\ds v),D^{s}w}-s\LR{\ds u\HT D^{s}\ds w,D^{s}w}|\\
&+|\LR{\HT D^{s}\ds(u\ds u-v\ds v),D^{s}w}-(s+1)\LR{\ds u\HT D^{s}\ds w,D^{s}w}|\\
&\le C\|w\|_{H^{s}}\{(\|u\|_{H^{s_{0}}}+\|v\|_{H^{s_{0}}})\|w\|_{H^{s}}
    +(\|u\|_{H^{s}}+\|v\|_{H^{s}})\|w\|_{H^{s_{0}}}\\
&\quad+\|w\|_{H^{s_{0}-2}}\|v\|_{H^{s+2}}+\|w\|_{H^{s_{0}-1}}\|v\|_{H^{s+1}}\},
}
where $w=u-v$.
\end{lem}

\begin{proof}
Adding and subtracting terms, we obtain
\EQQS{
&|\LR{D^{s}\ds(u\HT\ds w+w\HT\ds v),D^{s}w}-s\LR{\ds u\HT D^{s}\ds w,D^{s}w}|\\
&\le|\LR{P_{s}(u,\HT w)+P_{s}(w,\HT v),D^{s}w}|+|\LR{D^{s}\ds u\HT\ds w,D^{s}w}|\\
&\quad+|\LR{u\HT D^{s}\ds^{2}w+\ds u\HT D^{s}\ds w,D^{s}w}|+\frac{1}{2}|\LR{\HT\ds^{2}v,(D^{s}w)^{2}}|\\
&\quad+|\LR{w\HT D^{s}\ds^{2}v,D^{s}w}|+(s+1)|\LR{\ds w\HT D^{s}\ds v,D^{s}w}|,\\
&|\LR{\HT D^{s}\ds(u\ds w+w\ds v),D^{s}w}-(s+1)\LR{\ds u\HT D^{s}\ds w,D^{s}w}|\\
&\le|\LR{Q_{s}(u,w)+Q_{s}(w,v),D^{s}w}|+|\LR{u\HT D^{s}\ds^{2}w+\ds u\HT D^{s}\ds w,D^{s}w}|\\
&\quad+|\LR{w\HT D^{s}\ds^{2}v,D^{s}w}|+(s+2)|\LR{\ds w\HT D^{s}\ds v,D^{s}w}|
}
since we have
\[\LR{\ds w\HT D^{s}\ds u+\ds v\HT D^{s}\ds w,D^{s}w}=\LR{\ds u\HT D^{s}\ds w+\ds w\HT D^{s}\ds v,D^{s}w}.\]
Note that
\EQQS{
|\LR{D^{s}\ds u\HT\ds w,D^{s}w}|
&=|\LR{D^{s}\ds w\HT\ds w,D^{s}w}+\LR{D^{s}\ds v\HT\ds w,D^{s}w}|\\
&\lesssim\|w\|_{H^{s_{0}}}\|w\|_{H^{s}}^{2}+\|w\|_{H^{s}}\|w\|_{H^{s_{0}-1}}\|v\|_{H^{s+1}}
}
by Lemma \ref{IBP}.
This together with Lemma \ref{comm.est.} and \ref{reduction} gives the desired inequality,
which completes the proof.
\end{proof}


\begin{defn}\label{def1}
Let $s\ge2$ and $a,b,c\ge0$.
Set
$\la(s')=-2((c_{1}+c_{2})s'+c_{2})/3$ for $s'\ge0$.
For $f,g\in H^{s}(\T)$ we define
\EQQ{
E_{s}(f,g;a)&:=a\|f-g\|^{2}+\|D^{s}(f-g)\|^{2}+\la(s)\int_{\T} f(\HT D^{s}(f-g))D^{s-2}\ds(f-g)dx,\\
E_{s}(f;b)&:=E_{s}(f,0;1)+b\|f\|^{4s+2}.
}
For $f,g\in L^{2}(\T)$ we define
\EQQ{
\tilde{E}(f,g;c):=c\|f-g\|_{H^{-1}}^{2}+\|f-g\|^{2}-\la(0)\int_{\T}f(\LR{D}^{-1}(f-g))(f-g)dx.
}
\end{defn}

\begin{lem}\label{comparison}
Let $s\ge s_{0}>5/2$ and $K>0$.
Then\\
(i) If $f,g\in H^{s}(\T)$ and $f$ satisfies $\|f\|\le K$,
then there exist $C=C(s,K)$ and $a=a(s,K)$ such that
\EQS{\label{comp.s}
\|f-g\|_{H^{s}}^{2}\le E_{s}(f,g;a)\le C\|f-g\|_{H^{s}}^{2}.
}
(ii) If $f\in H^{s}(\T)$, there exist $C=C(s)$ and $b=b(s)$ such that
\EQS{\label{comp.sg.}
\|f\|_{H^{s}}^{2}\le E_{s}(f;b)\le C(1+\|f\|^{4s})\|f\|_{H^{s}}^{2}
}
(iii) If $f,g\in L^{2}(\T)$ and $f$ satisfies $\|f\|\le K$, then there exist $c=c(K)$ and $C=C(K)$ such that
\EQS{\label{comp.0}
\frac{1}{2}\|f-g\|^{2}\le\tilde{E}(f,g;c)\le C\|f-g\|^{2}.
}
\end{lem}

\begin{proof}
We see from Lemma \ref{G.N.} and the Young inequality that
\EQQS{
\int_{\T}|f(\HT D^{s}(f-g))D^{s-2}\ds(f-g)|dx
&\le\|f\|\|D^{s}(f-g)\|\|D^{s-2}\ds(f-g)\|_{\I}\\
&\le C\|f-g\|^{1/2s}\|D^{s}(f-g)\|^{2-1/2s}\\
&\le C\|f-g\|^{2}+\frac{1}{2}\|D^{s}(f-g)\|^{2}.
}
Choosing $a>0$ so that $a-C\ge1/2$,
we obtain the left hand side of (\ref{comp.s}).
The right hand side of (\ref{comp.s}) follows immediately, which shows $(i)$.

Next we prove (\ref{comp.sg.}).
A similar argument to the proof of (\ref{comp.s}) yields that
\[\int_{\T}|f(\HT D^{s}f)D^{s-2}\ds f|dx\le C\|f\|^{4s+2}+\frac{1}{2}\|D^{s}f\|^{2}.\]
Choosing $b>0$ so that $b-C>1/2$, we obtain (\ref{comp.sg.}).
The proof of $(iii)$ is identical with that of $(i)$.
\end{proof}

In what follows, we simply write $E_{s}(f,g):=E_{s}(f,g;a)$, $E_{s}(f):=E_{s}(f;b)$
and $\tilde{E}_{s}(f,g):=\tilde{E}_{s}(f,g;c)$, where $a,b$ and $c$ are defined by Lemma \ref{comparison}.

\begin{defn}\label{BS}
Let $s\ge0$, $f\in H^{s}(\T)$ and $\ga\in(0,1)$.
And let $\rho\in C_{0}^{\I}(\R)$ satisfy $\supp\rho\subset[-2,2]$,
$0\le\rho\le1$ on $\R$ and $\rho\equiv1$ on $[-1,1]$.
We put
\[\ha{J_{\ga}f}(k):=\rho(\ga k)\hat{f}(k).\]
\end{defn}

For the proof of the following lemma, see Remark 3.5 in \cite{ErTzi}.

\begin{lem}\label{lem_BS}
Let $s\ge 0$, $\al\ge0$, $\ga\in(0,1)$ and $f\in H^s(\T)$.
Then, $J_{\ga} f \in H^\I(\T)$ satisfies
\EQQ{
&\|J_{\ga} f -f\|_{H^s} \to 0 \quad(\ga\to 0), \quad\|J_{\ga} f -f \|_{H^{s-\al}} \lec \ga^{\al} \|f\|_{H^s},\\
&\|J_{\ga}f\|_{H^{s-\al}}\le \|f\|_{H^{s-\al}}, \quad\|J_{\ga} f\|_{H^{s+\al}} \lec \ga^{-\al}\|f\|_{H^s}.
}
\end{lem}

We employ the parabolic regularization on the problem (\ref{BO1})-(\ref{BO2}):
\EQS{
\label{BO1pr}
&\dt u-\ds^{3}u+u^{2}\ds u+c_{1}\ds(u\HT\ds u)+c_{2}\HT\ds(u\ds u)=-\ga D^{5/2}u,\\
\label{BO2pr}
&u(0,x)=\vp(x),
}
where $(t,x)\in[0,\I)\times\T$ and $\ga\in(0,1)$.
In what follows, we only consider $t\ge0$.
In the case $t\le0$, we only need to replace $-\ga D^{5/2}u$ with $\ga D^{5/2}u$ in (\ref{BO1pr}).

\begin{prop}\label{para.regu.}
Let $s\ge2$ and $\ga\in(0,1)$.
For any $\vp\in H^{s}(\T)$,
there exist $T_{\ga}\in(0,\I]$ and the unique solution $u\in C([0,T_{\ga}),H^{s}(\T))$
to the IVP (\ref{BO1pr})--(\ref{BO2pr}) on $[0,T_{\ga})$ such that
(i) $\liminf_{t\to T_{\ga}}\|u(t)\|_{H^{2}}=\I$ or (ii) $T_{\ga}=\I$ holds.
Moreover, $u$ satisfies
\EQS{\label{regularity}
u\in C((0,T_{\ga}),H^{s+\al}(\T)),\quad\forall\al>0.
}
\end{prop}

\begin{proof}
This follows from the standard argument, for expamle, see Proposition 2.8 in \cite{Tsugawa17},
but we reproduce the proof here for the sake of completeness.
First we consider the case $s=2$.
For simplicity, set $F(u)=-u^{2}\ds u-c_{1}\ds(u\HT\ds u)-c_{2}\HT\ds(u\ds u)$.
Let $U_{\ga}(t)$ be the linear propagator of the linear part of (\ref{BO1pr}), i.e.,
\[U_{\ga}(t)\vp={\F}^{-1}[e^{-i\xi^{3}t-\ga|\xi|^{5/2}t}\hat{\vp}]\]
for a function $\vp$.
Note that
\EQS{\label{eq2.7}
\|D^{\al}U_{\ga}(t)\vp\|\le\frac{C(\al)}{(\ga t)^{2\al/5}}\|\vp\|\quad
  {\rm and}\quad
  \|U_{\ga}(t)\vp\|_{H^{\al}}\le C(\al)(1+(\ga t)^{-2\al/5})\|\vp\|
}
for $t>0$ and $\al>0$.
We show the map
\[\Ga(u(t))=U_{\ga}(t)\vp+\int_{0}^{t}U_{\ga}(t-\tau)F(u)d\tau\]
is a contraction on the ball
\[B_{r}=\left\{u\in C([0,T];H^{2}(\T));\|u\|_{X}:=\sup_{t\in[0,T]}\|u(t)\|_{H^{2}}\le r\right\},\]
where $r>0$ and $T$ will be chosen later (which is sufficiently small and depends only on $\|\vp\|_{H^{2}}$ and $\ga$).
Set $r=2\|\vp\|_{H^{s}}$.
We show that $\Gamma$ maps from $B_{r}$ to $B_{r}$.
Let $u\in B_{r}$.
Obviously,
\[\|\Ga(u(t))\|_{H^{2}}\le\|\vp\|_{H^{2}}+\int_{0}^{t}\|U_{\ga}(t-t')F(u)\|_{H^{2}}dt'.\]
The Plancherel theorem implies that
\EQQS{
\|U_{\ga}(t-t')\ds u^{3}\|_{H^{2}}
&=\|\LR{\xi}^{2}|\xi|e^{-\ga(t-t')|\xi|^{5/2}}{\F}u^{3}\|_{l^{2}}\\
&\lesssim\ga^{-2/5}(t-t')^{-2/5}\|u^{3}\|_{H^{2}}
\lesssim\ga^{-2/5}(t-t')^{-2/5}\|\vp\|_{H^{2}}^{3}.
}
Similarly, we have
\[\|U_{\ga}(t-t')\HT\ds(u\ds u)\|_{H^{2}}\lesssim\ga^{-4/5}(t-t')^{-4/5}\|\vp\|_{H^{2}}^{2}.\]
On the other hand,
\begin{align*}
\|U_{\ga}(t-t')\ds(u\HT\ds u)\|_{H^{2}}
&\lesssim(1+\ga^{-4/5}(t-t')^{-4/5})\|\vp\|_{H^{2}}^{2}.
\end{align*}
It then follows that
\EQQS{
&\sup_{t\in[0,T]}\|\Ga(u(t))\|_{H^{2}}\\
&\le\|\vp\|_{H^{2}}
   +C\{\|\vp\|_{H^{2}}^{2}\ga^{-2/5}T^{3/5}+\|\vp\|_{H^{2}}(T+\ga^{-4/5}T^{1/5})\}\|\vp\|_{H^{2}}
\le2\|\vp\|_{H^{2}}
}
for sufficiently small $T=T(\|\vp\|_{H^{2}},\ga)>0$ and any $u\in B_{r}$.
By a similar argument, we can show that $\|\Ga(u)-\Ga(v)\|_{X}\le2^{-1}\|u-v\|_{X}$ when $u,v\in B_{r}$.
Therefore, $\Ga$ is a contraction map from $B_{r}$ to $B_{r}$, which implies that there exists $u\in B_{r}$ such that
$u=\Gamma(u)$ on $[0,T]$.
Since $\|u(T)\|_{H^{2}}$ is finite, we can repeat the argument above with initial data $u(T)$ to obtain the solution
on $[T,T+T']$.
Iterating this process, we can extend the solution on $[0,T_{\ga})$
where $T_{\ga}=\I$ or $\liminf_{t\to T_{\ga}}\|u(t)\|_{H^{2}}=\I$ holds.

Next, we consider the case $s>2$.
The solution obtained by the argument above satisfies
\begin{align}\label{eq3.60}
u(t)=U_{\ga}(t)\vp+\int_{0}^{t}U_{\ga}(t-t')F(u)dt'.
\end{align}
Note that
\[\|U_{\ga}(t-t')\ds u^{3}\|_{H^{s}}
  \lesssim\ga^{-2/5}(t-t')^{-2/5}\|u^{3}\|_{H^{s}}
  \lesssim\ga^{-2/5}(t-t')^{-2/5}\|\vp\|_{H^{2}}^{2}\|\vp\|_{H^{s}}.\]
We can estimate the other nonlinear terms in the same manner as above.
It then follows that
\EQQS{
\sup_{t\in[0,T]}\|u(t)\|_{H^{s}}
&\le\|\vp\|_{H^{s}}
   +C\{\|\vp\|_{H^{2}}^{2}\ga^{-2/5}T^{3/5}+\|\vp\|_{H^{2}}(T+\ga^{-4/5}T^{1/5})\}\|\vp\|_{H^{s}}\\
&\le2\|\vp\|_{H^{s}}
}
for sufficiently small $T=T(\|\vp\|_{H^{2}},\ga)>0$.
By using (\ref{eq3.60}), we also obtain $u\in C([0,T];H^{s}(\T))$.
Since $\|u(T)\|_{H^{s}}$ is finite, we can repeat the argument above with initial data $u(T)$
to obtain $u\in C([T,T+T'];H^{s}(\T))$.
We can iterate this process as far as $\|u(t)\|_{H^{2}}<\I$.
Therefore, we obtain $u\in C([0,T_{\ga});H^{s}(\T))$.
We omit the proof of the uniqueness since it follows from a standard argument.
Let $0<\de<T_{\ga}/2$.
We see from (\ref{eq2.7}) and (\ref{eq3.60}) that $u\in C([\delta,T_{\ga});H^{s+1/4}(\T))$.
The same argument as above with the initial data
$u(\de)\in H^{s+1/4}(\T)$ shows that $u\in C([\de+\de/2,T_{\ga});H^{s+1/2}(\T))$.
Iterating this procedure, we obtain (\ref{regularity}) since $\de$ is arbitrary, which completes the proof.
\end{proof}

\section{energy estimate}

In this section, we obtain an {\it a priori} estimate of the solution of (\ref{BO1}),
which is important to have the time $T$ independent of $\ga$.

\begin{prop}\label{ene.est.}
Let $s\ge s_{0}>5/2$,
$\ga\in(0,1)$, $\vp\in H^{s}(\T)$.
Let 
$T_{\ga}>0$ and let $u\in C([0,T_{\ga}),H^{s}(\T))\cap C((0,T_{\ga});H^{s+3}(\T))$
be the solution to (\ref{BO1pr})--(\ref{BO2pr}),
both of which are obtained by Proposition \ref{para.regu.}.
Then, there exist $T=T(s_{0},\|\vp\|_{H^{s_{0}}})>0$ and $C=C(s,s_{0},\|\vp\|_{H^{s_{0}}})>0$ such that
\begin{align}\label{single1}
T_{\ga}\ge T,\quad\sup_{t\in[0,T]}E_{s}(u(t))\le CE_{s}(\vp),\quad\frac{d}{dt}E_{s}(u(t))\le CE_{s}(u(t))
\end{align}
on $[0,T]$, where $T$ (resp. $C$) is monotone decreasing (resp. increasing) with $\|\vp\|_{H^{s_{0}}}$.
\end{prop}

Before proving Proposition \ref{ene.est.}, we give the following lemma.
\begin{lem}\label{pre.ene.}
Let $s\ge s_{0}>5/2$,
$\ga\in[0,1)$, $T>0$, $u\in C([0,T],H^{s}(\T))\cap C((0,T];H^{s+3}(\T))$ satisfy (\ref{BO1pr}) on $[0,T]\times\T$
and $\sup_{t\in[0,T]}E_{s_{0}}(u(t))\le K$ for $K>0$.
Then, there exists $C=C(s,s_{0},K)>0$ such that
\EQQS{
\frac{d}{dt}E_{s}(u(t))\le CE_{s}(u(t))
}
on $[0,T]$.
\end{lem}

\begin{proof}
First observe that
\EQQS{
\frac{d}{dt}\|u(t)\|^{2}
&=2\LR{\ds^{3}u-u^{2}\ds u-c_{1}\ds(u\HT\ds u)-c_{2}\HT\ds(u\ds u),u}\\
&\lesssim\|u(t)\|_{H^{1}}^{2}\le\|u(t)\|_{H^{s}}^{2}.
}
We can estimate the time derivative of $\|u(t)\|^{4s+2}$ in a similar manner.
Next we consider
\EQQS{
\frac{d}{dt}\|D^{s}u\|^{2}
&=2\LR{D^{s}\ds^{3}u,D^{s}u}
  -2\LR{D^{s}(u^{2}\ds u),D^{s}u}
  -2c_{1}\LR{D^{s}\ds(u\HT\ds u),D^{s}u}\\
&\quad-2c_{2}\LR{\HT D^{s}\ds(u\ds u),D^{s}u}-2\ga\LR{D^{s+5/2}u,D^{s}u}\\
&=:R_{1}+R_{2}+R_{3}+R_{4}+R_{5}.
}
It is clear that $R_{1}=0$.
We have
\EQQS{
|R_{2}|\le2|\LR{[D^{s},u^{2}]\ds u,D^{s}u}|+2|\LR{u^{2}D^{s}\ds u,D^{s}u}|\lesssim\|u\|_{H^{s}}^{2}
}
by $(i)$ of Lemma \ref{comm.est.2} and Lemma \ref{IBP}.
Lemma \ref{DL1} with $v=0$ shows that
\EQQS{
|R_{3}+2c_{1}s\LR{\ds u\HT D^{s}\ds u,D^{s}u}|+|R_{4}+2c_{2}(s+1)\LR{\ds u\HT D^{s}\ds u,D^{s}u}|
\lesssim\|u\|_{H^{s}}^{2}
}
Finally, we have $R_{5}=-2\ga\|D^{s+5/4}u\|^{2}$.
Therefore, we have
\EQS{\label{est.3.3}
\frac{d}{dt}\|D^{s}u\|^{2}
\le C\|u\|_{H^{s}}^{2}+3\la(s)\int_{\mathbb{T}}\ds u(\HT D^{s}\ds u)D^{s}udx-2\ga\|D^{s+5/4}u\|^{2},
}
where $\la(s)$ is defined in Definition \ref{def1}.
Next we evaluate the correction term.
We put
\EQQS{
&\frac{d}{dt}\LR{u\HT D^{s}u,D^{s-2}\ds u}\\
&=\LR{\dt u\HT D^{s}u,D^{s-2}\ds u}+\LR{u\HT D^{s}\dt u,D^{s-2}\ds u}
 +\LR{u\HT D^{s}u,D^{s-2}\ds\dt u}\\
&=:R_{6}+R_{7}+R_{8}.
}
Moreover, we set
\EQQS{
R_{6}
&=\LR{\ds^{3}u\HT D^{s}u,D^{s-2}\ds u}-\LR{u^{2}\ds u\HT D^{s}u,D^{s-2}\ds u}\\
&\quad-c_{1}\LR{\ds(u\HT\ds u)\HT D^{s}u,D^{s-2}\ds u}-c_{2}\LR{(\HT\ds(u\ds u))\HT D^{s}u,D^{s-2}\ds u}\\
&\quad-\ga\LR{D^{5/2}u\HT D^{s}u,D^{s-2}\ds u}
=:R_{61}+R_{62}+R_{63}+R_{64}+R_{65}.
}
And we set
\EQQS{
R_{7}
&=\LR{u\HT D^{s}\ds^{3}u,D^{s-2}\ds u}-\LR{u\HT D^{s}(u^{2}\ds u),D^{s-2}\ds u}\\
&\quad-c_{1}\LR{u\HT D^{s}\ds(u\HT\ds u),D^{s-2}\ds u}+c_{2}\LR{u D^{s}\ds(u\ds u),D^{s-2}\ds u}\\
&\quad-\ga\LR{u\HT D^{s+5/2}u,D^{s-2}\ds u}
=:R_{71}+R_{72}+R_{73}+R_{74}+R_{75}.
}
Finally, we set
\EQQS{
R_{8}
&=\LR{u\HT D^{s}u,D^{s-2}\ds^{4}u}-\LR{u\HT D^{s}u,D^{s-2}\ds(u^{2}\ds u)}\\
&\quad+c_{1}\LR{u\HT D^{s}u,D^{s}(u\HT\ds u)}+c_{2}\LR{u\HT D^{s}u,\HT D^{s}(u\ds u)}\\
&\quad-\ga\LR{u\HT D^{s}u,D^{s+1/2}\ds u}
=:R_{81}+R_{82}+R_{83}+R_{84}+R_{85}.
}
Lemma \ref{good1} shows that
\EQQS{
R_{61}+R_{71}+R_{81}
=&3\LR{\ds u\HT D^{s}\ds u,D^{s-2}\ds^{2}u}
=-3\LR{\ds u\HT D^{s}\ds u,D^{s}u},
}
which cancels out the second term in the right hand side in (\ref{est.3.3}) by multiplying $\la(s)$.
It is easy to see that $|R_{62}|+|R_{63}|+|R_{64}|\lesssim\|u\|_{H^{s}}^{2}$.
By $(i)$ of Lemma \ref{comm.est.2}, we have $|R_{72}|+|R_{82}|\lesssim\|u\|_{H^{s}}^{2}$.
We see from $(ii)$ of Lemma \ref{good2} that $|R_{73}|\lesssim\|u\|_{H^{s}}^{2}$.
Lemma \ref{IBP} and $(i)$ of Lemma \ref{comm.est.2} give $|R_{74}|+|R_{83}|\lesssim\|u\|_{H^{s}}^{2}$.
For $R_{84}$, it follows from $(i)$ of Lemma \ref{good2} that $|R_{84}|\lesssim\|u\|_{H^{s}}^{2}$.
Finally, we estimate $R_{65},R_{75}$ and $R_{85}$.
Lemma \ref{G.N.} implies that
\[\|D^{s-2}\ds u\|_{\I}\le C\|D^{s-2}u\|^{1/4}\|D^{s}u\|^{3/4}
  \le C\|u\|^{1-(4s-2)/(4s+5)}\|D^{s+5/4}u\|^{(4s-2)/(4s+5)}.\]
Then we have
\EQQS{
|R_{65}|
&\le\ga\|D^{5/2}u\|\|D^{s}u\|\|D^{s-2}\ds u\|_{\I}\\
&\le\ga C\|u\|^{1+2/(4s+5)}\|D^{s+5/4}u\|^{2-2/(4s+5)}
\le C\|u\|^{4s+7}+\frac{\ga^{1+1/4(s+1)}}{3}\|D^{s+5/4}u\|^{2}.
}
A similar argument yields
\EQQS{
|R_{75}|+|R_{85}|
\le C\|u\|^{4s+7}+C\|u\|^{2s+9/2}+\frac{2\ga^{1+1/4(s+1)}}{3}\|D^{s+5/4}u\|^{2}.
}
Therefore, the fact that $\ga\in[0,1)$ shows that
\[\frac{d}{dt}E_{s}(u(t))\le C\|u(t)\|_{H^{s}}^{2}\le CE_{s}(u(t))\]
on $[0,T]$.
Note that the implicit constant does not depend on $\ga$.
This completes the proof.
\end{proof}

Now, we are ready to prove Proposition \ref{ene.est.}.

\begin{proof}[Proof of Proposition \ref{ene.est.}]
Assume that the set $F=\{t\ge0; E_{s_{0}}(u(t))>2E_{s_{0}}(\vp)\}$ is not empty.
Set $T_{\ga}^{*}=\inf F$.
Note that $0<T_{\ga}^{*}\le T_{\ga}$ and $E_{s_{0}}(u(t))\le2E_{s_{0}}(\vp)$ on $[0,T_{\ga}^{*}]$.
Assume that there exists $t'\in[0,T_{\ga}^{*}]$ such that
$E_{s_{0}}(u(t'))>2E_{s_{0}}(\vp)$.
This implies that $t'\ge T_{\ga}^{*}$ by the definition of $T_{\ga}^{*}$.
Then we have $t'=T_{\ga}^{*}$.
Thus, $\sup_{t\in[0,T_{\ga}^{*}]}E_{s_{0}}(u(t))\le C(\|\vp\|_{H^{s_{0}}})$ by $(ii)$ of Lemma \ref{comparison}.
By Proposition \ref{pre.ene.}, there exists $C_{s}'=C(s,s_{0},\|\vp\|_{H^{s_{0}}})$ such that
\[\frac{d}{dt}E_{s}(u(t))\le C_{s}'E(u(t))\]
on $[0,T_{\ga}^{*}]$.
The Gronwall inequality gives that
\begin{align}\label{eq3.62}
E_{s}(u(t))\le E_{s}(\vp)\exp(C_{s}'t)
\end{align}
on $[0,T_{\ga}^{*}]$.
Here, we put $T=\min\{(2C_{s_{0}}')^{-1},T_{\ga}^{*}\}$.
Then (\ref{eq3.62}) with $s=s_{0}$ shows that
\[E_{s_{0}}(u(t))\le E_{s_{0}}(\vp)\exp(2^{-1})<2E_{s_{0}}(\vp),\]
on $[0,T]$.
By the definition of $T_{\ga}^{*}$ and the continuity of $E_{s_{0}}(u(t))$,
we obtain $0<T=(2C_{s_{0}}')^{-1}<T_{\ga}^{*}\le T_{\ga}$.
If $F$ is empty, then we have $T_{\ga}^{*}=T_{\ga}=\I$.
In particular, we can take $T=(2C_{s_{0}}')^{-1}<\I$, which concludes the proof.
\end{proof}


\section{uniqueness, persistence and continuous dependence}

In this section, we prove Theorem \ref{main}.
We first show the existence of the solution of (\ref{BO1}) by the limiting procedure.
We also prove the uniqueness and the persistence property $u\in C([0,T];H^{s}(\T))$.
Then we estiamte difference between two solutions of (\ref{BO1BS})--(\ref{BO2BS}) in $H^{s}(\T)$,
which is essential to show the continuous dependence.

\begin{lem}\label{en.dif.L2}
Let $s\ge s_{0}>5/2$, $\ga_{j}\in(0,1)$, $T>0$, $u_{j}\in C([0,T];H^{s}(\T))\cap C((0,T];H^{s+1}(\T))$
satisfy (\ref{BO1pr}) with $\ga=\ga_{j}$
on $[0,T]\times\T$ and $\sup_{t\in[0,T]}\|u_{j}(t)\|_{H^{s_{0}}}\le K$ for $K>0$, $j=1,2$.
Then there exists $C=C(K,s)$ such that
\begin{align}\label{E.L2}
\frac{d}{dt}\tilde{E}(u_{1},u_{2})\le C(\tilde{E}(u_{1},u_{2})+\max\{\ga_{1}^{2},\ga_{2}^{2}\})
\end{align}
on $[0,T]$.
\end{lem}

\begin{proof}
Set $w:=u_{1}-u_{2}$ so that $w$ satisfies the following equation:
\EQS{\label{eq.for.w2}
\begin{aligned}
&\dt w-\ds^{3}w+\frac{1}{3}\ds\{(u_{1}^{2}+u_{1}u_{2}+u_{2}^{2})w\}\\
&+\frac{c_{1}}{2}\ds(w\HT\ds z)+\frac{c_{1}}{2}\ds(z\HT\ds w)
 +\frac{c_{2}}{2}\HT\ds(w\ds z)+\frac{c_{2}}{2}\HT\ds(z\ds w)\\
&=-\ga_{1}D^{5/2}w-(\ga_{1}-\ga_{2})D^{5/2}u_{2},
\end{aligned}
}
where $z=u_{1}+u_{2}$.
By the presence of the operator $\LR{D}^{-1}$, we can easily obtain
\[\frac{d}{dt}\|\LR{D}^{-1}w\|^{2}\lesssim\|w\|^{2}+\max\{\ga_{1}^{2},\ga_{2}^{2}\}.\]
Next, we estimate the $L^{2}$-norm of $w$.
Set
\EQQS{
\frac{d}{dt}\|w\|^{2}
&=2\LR{\ds^{3}w,w}
 -\frac{2}{3}\LR{\ds\{(u_{1}^{2}+u_{1}u_{2}+u_{2}^{2})w\},w}
 -c_{1}\LR{\ds(w\HT\ds z),w}\\
&\quad-c_{1}\LR{\ds(z\HT\ds w),w}
 -c_{2}\LR{\HT\ds(w\ds z),w}-c_{2}\LR{\HT\ds(z\ds w),w}\\
&\quad-2\ga_{1}\LR{D^{5/2}w,w}-2(\ga_{1}-\ga_{2})\LR{D^{5/2}u_{2},w}\\
&=:R_{9}+R_{10}+R_{11}+R_{12}+R_{13}+R_{14}+R_{15}+R_{16}.
}
Again, it is clear that $R_{9}=0$.
By Lemma \ref{IBP}, we have $|R_{10}|+|R_{11}|\lesssim\|w\|^{2}$.
Note that
\EQQS{
&\LR{[\HT,\ds z]\ds w,w}+\LR{[\HT,z]\ds^{2}w,w}\\
&=\LR{\HT(\ds z\ds w),w}-\LR{\ds z\HT\ds w,w}+\LR{\HT(z\ds^{2}w),w}-\LR{z\HT\ds^{2}w,w}\\
&=\LR{\ds(\ds z\HT w),w}-\LR{\ds z\HT\ds w,w}-\LR{\ds^{2}(z\HT w),w}-\LR{z\HT\ds^{2}w,w}\\
&=-2\LR{\ds(z\HT\ds w),w}.
}
Then Lemma \ref{comm.est.4} shows that $|R_{12}|+|R_{14}|\lesssim\|w\|^{2}$.
We can reduce $R_{13}$ to
\[R_{13}=-2c_{2}\LR{\ds u_{1}\HT\ds w,w}-c_{2}\LR{\ds w\HT\ds w,w}\]
since $z=2u_{1}-w$.
The last term in the right hand side can be bounded by $\lesssim\|w\|^{2}$ by using Lemma \ref{IBP}.
Observe that $R_{15}=-\ga_{1}\|D^{5/4}w\|^{2}\le0$
and that $|R_{16}|\lesssim\|w\|^{2}+\max\{\ga_{1}^{2},\ga_{2}^{2}\}$.
Therefore, we have
\[\frac{d}{dt}\|w\|^{2}\le C\|w\|^{2}+3\la(0)\int_{\T}\ds u_{1}(\HT\ds w)wdx+\max\{\ga_{1}^{2},\ga_{2}^{2}\}.\]
The correction term in $\ti{E}$ cannot exactly cancel out the second term,
but Lemma \ref{freq.est.} shows that the difference is harmless.
Set
\EQQS{
\frac{d}{dt}\LR{u_{1}\LR{D}^{-1}w,w}
&=\LR{\dt u_{1}\LR{D}^{-1}w,w}+\LR{u_{1}\LR{D}^{-1}\dt w,w}+\LR{u_{1}\LR{D}^{-1}w,\dt w}\\
&=:R_{17}+R_{18}+R_{19}.
}
Moreover, we set $R_{171}=\LR{\ds^{3}u_{1}\LR{D}^{-1}w,w}$ and set
\EQQS{
R_{18}
&=\LR{u_{1}\LR{D}^{-1}\ds^{3}w,w}-\frac{1}{3}\LR{u_{1}\LR{D}^{-1}\ds\{(u_{1}^{2}+u_{1}u_{2}+u_{2}^{2})w\},w}\\
&\quad-\frac{c_{1}}{2}\LR{u_{1}\LR{D}^{-1}\ds(w\HT\ds z),w}-\frac{c_{1}}{2}\LR{u_{1}\LR{D}^{-1}\ds(z\HT\ds w),w}\\
&\quad-\frac{c_{2}}{2}\LR{u_{1}\LR{D}^{-1}\HT\ds(w\ds z),w}-\frac{c_{2}}{2}\LR{u_{1}\LR{D}^{-1}\HT\ds(z\ds w),w}\\
&\quad-\ga_{1}\LR{u_{1}\LR{D}^{-1}D^{5/2}w,w}-(\ga_{1}-\ga_{2})\LR{u_{1}\LR{D}^{-1}D^{5/2}u_{2},w}\\
&=:R_{181}+R_{182}+R_{183}+R_{184}+R_{185}+R_{186}+R_{187}+R_{188}
}
We set $R_{19k}$ for $k=1,\dots,8$ in the same manner as above.
Lemma \ref{good1} shows that
\EQQS{
R_{171}+R_{181}+R_{191}
=-3\LR{\ds u_{1}\LR{D}^{-1}\ds^{2}w,w}-3\LR{\ds^{2}u_{1}\LR{D}^{-1}\ds w,w},
}
which together with Lemma \ref{freq.est.} shows that $|R_{13}-\la(0)(R_{171}+R_{181}+R_{191})|\lesssim\|w\|^{2}$.
It is easy to see that
\EQQS{
|\LR{(u_{1}^{2}\ds u_{1}+c_{1}\ds(u_{1}\HT\ds u_{1})
  +c_{2}\HT\ds(u_{1}\ds u_{1})+\ga_{1}D^{5/2}u_{1})\LR{D}^{-1}w,w}|
\lesssim\|w\|^{2}.
}
We have $|R_{182}|+|R_{183}|+|R_{185}|+|R_{192}|+|R_{193}|+|R_{195}|\lesssim\|w\|^{2}$
because of the presence of the operator $\LR{D}^{-1}$.
In order to handle $R_{184},R_{186},R_{194}$ and $R_{196}$,
we see from Lemma \ref{freq.est.} and $(i)$ of Lemma \ref{comm.est.2} that
\EQQS{
&|R_{196}|
=\left|-\frac{c_{2}}{2}\LR{u_{1}\LR{D}^{-1}w,\HT\ds^{2}(zw)}
 +\frac{c_{2}}{2}\LR{u_{1}\LR{D}^{-1}w,\HT\ds(\ds zw)}\right|\\
&\lesssim|\LR{u_{1}\LR{D}^{-1}\ds w,(\HT\ds+\LR{D}^{-1}\ds^{2})(zw)}|
         +|\LR{u_{1}\LR{D}^{-1}\ds w,\LR{D}^{-1}\ds^{2}(zw)}|+\|w\|^{2}\\
&\lesssim|\LR{u_{1}\LR{D}^{-1}\ds w,[\LR{D}^{-1}\ds^{2},z]w}|+|\LR{u_{1}z\LR{D}^{-1}\ds w,\LR{D}^{-1}\ds^{2}w}|+\|w\|^{2}
\lesssim\|w\|^{2}.
}
We can obtain $|R_{184}|+|R_{186}|+|R_{194}|\lesssim\|w\|^{2}$ from a similar argument.
Finally, it is easy to see that
$|R_{187}|+|R_{188}|+|R_{197}|+|R_{198}|\lesssim\|w\|^{2}+\max\{\ga_{1}^{2},\ga_{2}^{2}\}$.
Summing these estimates above and applying $(iii)$ of Lemma \ref{comparison}, we obtain (\ref{E.L2}),
which concludes the proof.
\end{proof}

Now we obtain the solution to (\ref{BO1})--(\ref{BO2}).
Let $\vp\in H^{s}(\T)$ and let $\ga_{1},\ga_{2}\in(0,1)$.
Let $u_{\ga_{j}}$ be the solution to (\ref{BO1pr})--(\ref{BO2pr}) with $\ga=\ga_{j}$ for $j=1,2$,
obtained by Proposition \ref{para.regu.}.
Note that $\tilde{E}(u_{\ga_{1}}(0),u_{\ga_{2}}(0))=\tilde{E}(\vp,\vp)=0$.
Proposition \ref{ene.est.} shows that there exists $T=T(s_{0},\|\vp\|_{H^{s_{0}}})$
such that (\ref{single1}) holds.
We see from $(iii)$ of Lemma \ref{comparison} and Lemma \ref{en.dif.L2} that
\begin{align*}
\sup_{t\in[0,T]}\|u_{\ga_{1}}(t)-u_{\ga_{2}}(t)\|^{2}
&\le\sup_{t\in[0,T]}\tilde{E}(u_{\ga_{1}}(t),u_{\ga_{2}}(t))
\le C\max\{\ga_{1}^{2},\ga_{2}^{2}\}\to0
\end{align*}
as $\ga_{1},\ga_{2}\to+0$.
This implies that there exists $u\in C([0,T];L^{2}(\T))$ such that
\[u_{\ga}\to u\ \ {\rm in}\ \ C([0,T];L^{2}(\T))\ \ {\rm as}\ \ \ga\to0.\]
The above convergence can be verified in $C([0,T];H^{r}(\T))$ for any $r<s$
by interpolating with $L^{\I}([0,T];H^{s}(\T))$.
It is clear that $u$ satisfies (\ref{BO1})--(\ref{BO2}) on $[0,T]$.

For the proof of the following uniqueness result, see Thorem 6.22 in \cite{Iorio}.


\begin{lem}[Uniqueness]\label{uniq.}
Let $\delta>0$ and $\e>0$,
$u_{j}\in L^{\I}([0,\delta];H^{5/2+\e}(\T))$ satisfy (\ref{BO1}) on $[0,\delta]$ with $u_{1}(0)=u_{2}(0)$ and satisfy
\[u_{j}\in C([0,\delta];H^{2}(\T))\cap C^{1}([0,\delta];H^{-1}(\T))\]
for $j=1,2$.
Then $u_{1}\equiv u_{2}$ on $[0,\delta]$.
\end{lem}

It remains to show the persistent property, i.e., $u\in C([0,T];H^{s}(\T))$ and the continuous dependence.
In what follows, we employ the Bona-Smith approximation argument.
We consider the following initial value problem:
\EQS{\label{BO1BS}
&\dt u-\ds^{3}u+u^{2}\ds u+c_{1}\ds(u\HT\ds u)+c_{2}\HT\ds(u\ds u)=0,\quad x\in\T,\\
\label{BO2BS}
&u(0,x)=J_{\ga}\vp(x),
}
where $J_{\ga}\vp$ is defined in Definition \ref{BS}.
Let $s\ge s_{0}>5/2$, $\vp\in H^{s}(\T)$ and $\epsilon>0$.
Lemma \ref{lem_BS} shows that $J_{\ga}\vp\in H^{\I}(\T)$.
Let $u_{\ga}\in C([0,T_{\ga});H^{s+3+\epsilon}(\T))$ be the solution (\ref{BO1pr}) with the initial data $J_{\ga}\vp$
obtained by Proposition \ref{para.regu.}.
Lemma \ref{lem_BS} and Proposition \ref{ene.est.} imply that there exists
$T=T(s_{0},\|\vp\|_{H^{s_{0}}})(\le T'=T(s_{0},\|\vp^{\ga}\|_{H^{s_{0}}}))$
such that (\ref{single1}) holds for $s+3+\epsilon$.
Lemma \ref{en.dif.L2} and the above argument show that
there exists $\tilde{u}\in C([0,T];H^{s+3}(\T))$ such that $\tilde{u}$ solves (\ref{BO1BS})--(\ref{BO2BS}).
Therefore, we have the following corollary:

\begin{cor}\label{en.dif.L2-2}
Let $s\ge s_{0}>5/2$, $T>0$,
$u_{j}\in C([0,T];H^{s+1}(\T))$ satisfy (\ref{BO1BS})
on $[0,T]\times\T$ and $\sup_{t\in[0,T]}\|u_{j}(t)\|_{H^{s}}\le K$ for $K>0$, $j=1,2$.
Then there exists $C=C(K,s_{0},s)$ such that
\begin{align}\label{E.L2-2}
\frac{d}{dt}\tilde{E}(u_{1}(t),u_{2}(t))\le C\tilde{E}(u_{1}(t),u_{2}(t))
\end{align}
on $[0,T]$.
\end{cor}

\begin{prop}\label{en.dif.s2}
Let $s\ge s_{0}>5/2$, $T>0$,
$u_{j}\in C([0,T];H^{s+3}(\T))$ satisfy (\ref{BO1BS}) on $[0,T]\times\T$
and $\sup_{t\in[0,T]}\|u_{j}(t)\|_{H^{s}}\le K$ for $K>0$, $j=1,2$.
Then there exists $C=C(s,s_{0},K)$ such that
\EQS{\label{ene.dif.s}
\begin{aligned}
\frac{d}{dt}E_{s}(u_{1}(t),u_{2}(t))
\le &C(\|u_{1}(t)-u_{2}(t)\|_{H^{s}}^{2}+\|u_{1}(t)-u_{2}(t)\|_{H^{s_{0}-1}}^{2}\|u_{2}\|_{H^{s+1}}^{2}\\
    &+\|u_{1}(t)-u_{2}(t)\|_{H^{s_{0}-2}}^{2}\|u_{2}\|_{H^{s+2}}^{2})
\end{aligned}}
on $[0,T]$.
\end{prop}

\begin{proof}
Set $w=u_{1}-u_{2}$ and $z=u_{1}+u_{2}$.
It is easy to see that
\[\frac{d}{dt}\|w\|^{2}\lesssim\|w\|_{H^{1}}^{2}\le\|w\|_{H^{s}}^{2}.\]
Set
\begin{align*}
\frac{d}{dt}\|D^{s}w\|^{2}
&=2\LR{D^{s}\ds^{3}w,D^{s}w}
 -2\LR{D^{s}(u_{1}^{2}\ds w),D^{s}w}-2\LR{D^{s}(zw\ds u_{2}),D^{s}w}\\
&\quad-2c_{1}\LR{D^{s}\ds(u_{1}\HT\ds u_{1}-u_{2}\HT\ds u_{2}),D^{s}w}\\
&\quad-2c_{2}\LR{\HT D^{s}\ds(u_{1}\ds u_{1}-u_{2}\ds u_{2}),D^{s}w}
=:R_{1}+R_{2}+R_{3}+R_{4}+R_{5}.
\end{align*}
It is easy to see that $R_{1}=0$ and $|R_{2}|\lesssim\|w\|_{H^{s}}^{2}$ by $(i)$ of Lemma \ref{comm.est.2}.
For $R_{3}$, we have $|R_{3}|\lesssim\|w\|_{H^{s}}^{2}+\|w\|_{H^{s_{0}-1}}^{2}\|u_{2}\|_{H^{s+1}}^{2}$.
Lemma \ref{DL1} shows that
\EQQS{
&|R_{4}+R_{5}-3\la(s)\LR{\ds u_{1}\HT D^{s}\ds w,D^{s}w}|\\
&\lesssim\|w\|_{H^{s}}^{2}+\|w\|_{H^{s_{0}-2}}^{2}\|u_{2}\|_{H^{s+2}}^{2}+\|w\|_{H^{s_{0}-1}}^{2}\|u_{2}\|_{H^{s+1}}^{2}.
}
Therefore, the time derivative of $\|D^{s}w\|^{2}$ yields
\EQS{\label{eq4.8}
\begin{aligned}
\frac{d}{dt}\|D^{s}w\|^{2}
&\le C\|w\|_{H^{s}}^{2}+C\|w\|_{H^{s_{0}-1}}^{2}\|u_{2}\|_{H^{s+1}}^{2}+C\|w\|_{H^{s_{0}-2}}^{2}\|u_{2}\|_{H^{s+2}}^{2}\\
&\quad+3\la(s)\int_{\mathbb{T}}\ds u_{1}(\HT D^{s}\ds w)D^{s}wdx.
\end{aligned}
}
Next, we evaluate the time derivative of the correction term.
Lemma \ref{good1} with $f=u_{1}$, $g=\HT D^{s}w$ and $h=D^{s-2}\ds w$ shows that
\EQQS{
&\LR{\ds^{3}u_{1}\HT D^{s}w,D^{s-2}\ds w}
  +\LR{u_{1}\HT D^{s}\ds^{3}w,D^{s-2}\ds w}
  +\LR{u_{1}\HT D^{s}w,D^{s-2}\ds^{4}w}\\
&=3\LR{\ds u_{1}\HT D^{s}\ds w,D^{s-2}\ds^{2}w}
=-3\LR{\ds u_{1}\HT D^{s}\ds w,D^{s}w}.
}
Multiplying by $\la(s)$, we can cancel out the last term in the right hand side in (\ref{eq4.8}).
On the other hand, it is easy to see that
\[\LR{(\dt u_{1}-\ds^{3}u_{1})\HT D^{s}w,D^{s-2}\ds w}\lesssim\|w\|_{H^{s}}^{2}.\]
We set
\EQQS{
&\LR{u_{1}\HT D^{s}(\dt w-\ds^{3}w),D^{s-2}\ds w}\\
&=-\frac{1}{3}\LR{u_{1}\HT D^{s}\ds\{(u_{1}^{2}+u_{1}u_{2}+u_{2}^{2})w\},D^{s-2}\ds w}\\
&\quad-c_{1}\LR{u_{1}\HT D^{s}\ds(u_{1}\HT\ds w),D^{s-2}\ds w}+c_{2}\LR{u_{1}D^{s}\ds(u_{1}\ds w),D^{s-2}\ds w}\\
&\quad-c_{1}\LR{u_{1}\HT D^{s}\ds(w\HT\ds u_{2}),D^{s-2}\ds w}+c_{2}\LR{u_{1}D^{s}\ds(w\ds u_{2}),D^{s-2}\ds w}\\
&=:R_{9}+R_{10}+R_{11}+R_{12}+R_{13}
}
and
\EQQS{
&\LR{u_{1}\HT D^{s}w,D^{s-2}\ds(\dt w-\ds^{3}w)}\\
&=\frac{1}{3}\LR{u_{1}\HT D^{s}w,D^{s}\{(u_{1}^{2}+u_{1}u_{2}+u_{2}^{2})w\}}
  +c_{1}\LR{u_{1}\HT D^{s}w,D^{s}(u_{1}\HT\ds w)}\\
&\quad+c_{2}\LR{u_{1}\HT D^{s}w,\HT D^{s}(u_{1}\ds w)}+c_{1}\LR{u_{1}\HT D^{s}w,D^{s}(w\HT\ds u_{2})}\\
&\quad+c_{2}\LR{u_{1}\HT D^{s}w,\HT D^{s}(w\ds u_{2})}
=:R_{14}+R_{15}+R_{16}+R_{17}+R_{18}.
}
By $(i)$ of Lemma \ref{comm.est.2}, we have $|R_{9}|+|R_{14}|\lesssim\|w\|_{H^{s}}^{2}$.
We see from $(ii)$ of Lemma \ref{good2} that $|R_{10}|\lesssim\|w\|_{H^{s}}^{2}$.
We also have $|R_{16}|\lesssim\|w\|_{H^{s}}^{2}$ by $(i)$ of Lemma \ref{good2}.
Similarly, we can obtain $|R_{11}|+|R_{15}|\lesssim\|w\|_{H^{s}}^{2}$.
On the other hand, by $(i)$ of Lemma \ref{comm.est.2} we have
$|R_{12}|+|R_{13}|+|R_{17}|+|R_{18}|\lesssim\|w\|_{H^{s}}^{2}+\|w\|_{H^{s_{0}-2}}^{2}\|u_{2}\|_{H^{s+1}}^{2}$.
Summing these estimates above, we obtain (\ref{ene.dif.s}) on $[0,T]$, which concludes the proof.
\end{proof}

Now, we can show the persistent property and the continuous dependence.
\begin{proof}[Proof of Theorem \ref{main}]
In what follows, without loss of generality, we may assume that $s_{0}$ is strictly smaller than $s$
since the assumption $\|\vp\|_{H^{s_{0}}}\le K$ is weaker than $\|\vp\|_{H^{s_{0}'}}\le K$
when $s_{0}<s_{0}'$.
First we prove the persistence property.
Let $0<\ga_{1}<\ga_{2}<1$.
Let $u_{\ga_{j}}\in C([0,T];H^{s+3}(\T))$ be the solution to (\ref{BO1BS})--(\ref{BO2BS})
with the initial data $J_{\ga}\vp$
for $\vp\in H^{s}(\T)$ and $j=1,2$.
Corollary \ref{en.dif.L2-2} with the Gronwall inequality shows that
\EQQS{
\sup_{t\in[0,T]}\|u_{\ga_{1}}(t)-u_{\ga_{2}}(t)\|^{2}
\le C\tilde{E}(u_{\ga_{1}}(0),u_{\ga_{2}}(0))
\le C\|J_{\ga_{1}}\vp-J_{\ga_{2}}\vp\|^{2}\le C\ga_{2}^{2s}
}
since $\ga_{1}<\ga_{2}$.
This together with the interpolation implies that
\EQQS{
\sup_{t\in[0,T]}\|u_{\ga_{1}}(t)-u_{\ga_{2}}(t)\|_{H^{\al}}^{2}
\le C\ga_{2}^{2(s-\al)}
}
for any $0\le\al<s$.
On the other hand, Lemma \ref{lem_BS} and \ref{pre.ene.} show that
\[\sup_{t\in[0,t]}\|u_{\ga_{2}}(t)\|_{H^{s+\al}}^{2}
  \le C\|J_{\ga_{2}}\vp\|_{H^{s+\al}}^{2}\le C\ga_{2}^{-2\al}\|\vp\|_{H^{s}}^{2}\]
for $\al\ge0$.
This together with the Gronwall inequality and Proposition \ref{en.dif.s2} implies that
\EQQS{
\sup_{t\in[0,T]}\|u_{\ga_{1}}(t)-u_{\ga_{2}}(t)\|_{H^{s}}^{2}
\lesssim\|J_{\ga_{1}}\vp-J_{\ga_{2}}\vp\|_{H^{s}}^{2}
         +\ga_{2}^{2(s-s_{0})}\to0
}
as $\ga_{2},\ga_{1}\to0$ since $\|J_{\ga_{1}}\vp-J_{\ga_{2}}\vp\|_{H^{s}}\to0$
as $\ga_{1},\ga_{2}\to0$.
Then, there exists $\tilde{u}\in C([0,T];H^{s}(\T))$ such that
\[u_{\ga}\to\tilde{u}\ \ {\rm in}\ \ C([0,T];H^{s}(\T))\quad{\rm as}\quad\ga\to0.\]
It is clear that the function $\tilde{u}$ coincides
with our solution $u\in C([0,T];H^{r}(\T))$ for $r<s$ to (\ref{BO1})--(\ref{BO2}),
which shows the persistence property.

Finally, we prove the continuous dependence, which is the only thing left to prove.
We will claim that
\EQS{\label{con.dep.}
\begin{aligned}
&\forall \vp\in H^{s}(\T),\forall\epsilon>0,\exists\delta>0,\forall\psi\in H^{s}(\T):\\
&\left[\|\vp-\psi\|_{H^{s}}<\delta\Rightarrow\sup_{t\in[0,T/2]}\|u(t)-v(t)\|_{H^{s}}<\epsilon\right],
\end{aligned}
}
where $u,v$ represent the solution to (\ref{BO1}) with initial data $\vp,\psi\in H^{s}(\T)$, respectively,
which are obtained by the above argument.
In (\ref{con.dep.}) we take the interval $[0,T/2]$ with $T$
as defined by Proposition \ref{ene.est.} to guarantee that
if $\|\vp-\psi\|_{H^{s}}<\delta$, then the solution $v(t)$ is defined in the time interval $[0,T/2]$.
Fix $\vp\in H^{s}(\T)$ and $\epsilon>0$.
Let $0<\ga_{1}<\ga_{2}<1$.
Assume that $\|\vp-\psi\|_{H^{s}}<\delta$, where $\delta>0$ will be chosen later.
Note that by the triangle inequality we have
\EQS{\label{eq3.74}
\begin{aligned}
&\sup_{t\in[0,T/2]}\|u(t)-v(t)\|_{H^{s}}\\
&\le\sup_{t\in[0,T/2]}\|u(t)-u^{\ga_{2}}(t)\|_{H^{s}}
   +\sup_{t\in[0,T/2]}\|u^{\ga_{2}}(t)-v^{\ga_{1}}(t)\|_{H^{s}}\\
&\quad+\sup_{t\in[0,T/2]}\|v^{\ga_{1}}(t)-v(t)\|_{H^{s}},
\end{aligned}
}
where $u^{\ga_{2}}$ and $v^{\ga_{1}}$ represent the solution to the IVP (\ref{BO1})
with the initial data $J_{\ga_{2}}\vp$ and $J_{\ga_{1}}\psi$, respectively.
First we handle the second term in the right hand side in (\ref{eq3.74}).
Again, the triangle inequality shows that
\[\|J_{\ga_{2}}\vp-J_{\ga_{1}}\psi\|_{H^{r}}
  \le\|J_{\ga_{2}}\vp-\vp\|_{H^{r}}+\|\vp-\psi\|_{H^{r}}+\|\psi-J_{\ga_{1}}\psi\|_{H^{r}}\]
for $r\le s$.
Proposition \ref{en.dif.s2} with $u_{1}=v^{\ga_{1}}$ and $u_{2}=u^{\ga_{2}}$ gives that
\EQQS{
&\sup_{t\in[0,T/2]}\|u^{\ga_{2}}(t)-v^{\ga_{1}}(t)\|_{H^{s}}\\
&\le C\|J_{\ga_{2}}\vp-\vp\|_{H^{s}}+C\delta+C\|\psi-J_{\ga_{1}}\psi\|_{H^{s}}
    +C\ga_{2}^{s-s_{0}}
    +C\ga_{2}^{-1}\delta^{1+1/s-s_{0}/s}\\
&\quad+C\ga_{2}^{-1}\|\psi-J_{\ga_{1}}\psi\|^{1+1/s-s_{0}/s}
    +C\ga_{2}^{s-s_{0}}
    +C\ga_{2}^{-2}\delta^{1+2/s-s_{0}/s}\\
&\quad+C\ga_{2}^{-2}\|\psi-J_{\ga_{1}}\psi\|^{1+2/s-s_{0}/s}.
}
Therefore, we choose $\ga_{2}>0$ so that
\[\sup_{t\in[0,T/2]}\|u(t)-u^{\ga_{2}}(t)\|_{H^{s}}+C\|J_{\ga_{2}}\vp-\vp\|_{H^{s}}
  +C\ga_{2}^{s-s_{0}}<\frac{\epsilon}{3},\]
Then we take $\delta>0$ such that
\[C(\delta+\ga_{2}^{-1}\delta^{1+1/s-s_{0}/s}+\ga_{2}^{-2}\delta^{1+2/s-s_{0}/s})<\frac{\epsilon}{3}\]
and finally for each $\psi\in H^{s}(\T)$ satisfying $\|\vp-\psi\|_{H^{s}}<\delta$
we take $\ga_{1}\in(0,\ga_{2})$ such that
\EQQS{
&\sup_{t\in[0,T/2]}\|v^{\ga_{1}}(t)-v(t)\|_{H^{s}}+C\|\psi-J_{\ga_{1}}\psi\|_{H^{s}}\\
&+C\ga_{2}^{-1}\|\psi-J_{\ga_{1}}\psi\|^{1+1/s-s_{0}/s}
  +C\ga_{2}^{-2}\|\psi-J_{\ga_{1}}\psi\|^{1+2/s-s_{0}/s}<\frac{\epsilon}{3}.
}
which completes the proof of (\ref{con.dep.}).
\end{proof}

\section*{Acknowlegdements}
The author would like to express his deep gratitude to Professor Kotaro Tsugawa for valuable comments and encouragement.

\end{document}